\documentclass[12pt,english]{extarticle}
\usepackage[T1]{fontenc}
\usepackage[latin9]{inputenc}
\usepackage{color}
\usepackage{babel}
\usepackage{float}
\usepackage{mathrsfs}
\usepackage{amsmath}
\usepackage{amsthm}
\usepackage{amssymb}
\usepackage{graphicx}
\usepackage{geometry}
\usepackage{microtype}
\usepackage[pdfusetitle,
 bookmarks=true,bookmarksnumbered=true,bookmarksopen=false,
 breaklinks=false,pdfborder={0 0 1},backref=false,colorlinks=true]
 {hyperref}
\hypersetup{
 linkcolor=blue, urlcolor=blue, citecolor=blue, pdfstartview={FitH}, hyperfootnotes=false, unicode=true}

\makeatletter
\numberwithin{equation}{section}
\numberwithin{figure}{section}
\theoremstyle{plain}
\newtheorem{thm}{\protect\theoremname}
\theoremstyle{definition}
\newtheorem{defn}[thm]{\protect\definitionname}
\theoremstyle{definition}
\newtheorem{example}[thm]{\protect\examplename}
\theoremstyle{plain}
\newtheorem{prop}[thm]{\protect\propositionname}
\theoremstyle{remark}
\newtheorem{rem}[thm]{\protect\remarkname}
\theoremstyle{plain}
\newtheorem{cor}[thm]{\protect\corollaryname}
\theoremstyle{plain}
\newtheorem{lem}[thm]{\protect\lemmaname}

\usepackage{lettrine} 
\let\myFoot\footnote
\renewcommand{\footnote}[1]{\myFoot{#1\vspace{3mm}}}
\usepackage{amsmath}
\usepackage{amssymb}
\usepackage{mlmodern}
\usepackage[T1]{fontenc}

\usepackage{tikz}
\usepackage{circuitikz}

\makeatother

\providecommand{\corollaryname}{Corollary}
\providecommand{\definitionname}{Definition}
\providecommand{\examplename}{Example}
\providecommand{\lemmaname}{Lemma}
\providecommand{\propositionname}{Proposition}
\providecommand{\remarkname}{Remark}
\providecommand{\theoremname}{Theorem}

\begin{document}
\title{A Numerical scheme to approximate the solution of the planar
Skorokhod embedding problem}
\author{Mrabet Becher \thanks{Institut Pr\'eparatoire aux Etudes des Ing\'enieurs de Monastir, Tunisia.} ,
Maher Boudabra \thanks{King Fahd University of Petroleum and Minerals, KSA.}
, Fathi Haggui \thanks{Institut Pr\'eparatoire aux Etudes des Ing\'enieurs de Monastir, Tunisia.}}
\maketitle
\begin{abstract}
We present a numerical framework to approximate the $\mu$-domain in the planar Skorokhod embedding problem (PSEP), recently appeared in \cite{gross2019}. Our approach investigates the continuity and convergence properties of the solutions with respect to the underlying distribution $\mu$. We establish that, under weak convergence of a sequence of probability measures $(\mu_n)$ with bounded support, the corresponding sequence of $\mu_n$-domains converges to the domain associated with $\mu$, limit of $(\mu_n)$. We derive explicit convergence results in the $L^1$ norm, supported by a generalization using the concept of $\alpha_p$-convergence. Furthermore, we provide practical implementation techniques, convergence rate estimates, and numerical simulations using various distributions. The method proves robust and adaptable, offering a concrete computational pathway for approximating $\mu$-domains in the PSEP.
\end{abstract}

\section{Introduction}

In 2019, the author R. Gross considered an interesting planar version
of the Skorokhod problem $\cite{gross2019}$, which was originally
formulated in 1961 in dimension one. We strongly refer the reader
to \cite{Obloj2004} for a concise survey of the linear version. 

The planar version studied by Gross is as follows: given a distribution
$\mu$ with zero mean and finite second moment, is there a simply
connected domain $U$ (containing the origin) such that if 
\[
Z_{t}=X_{t}+Y_{t}i
\]
 is a standard planar Brownian motion, then $X_{\tau}=\Re(Z_{\tau})$
has the distribution $\mu$, where $\tau$ is the exit time from $U$?
Gross provided an affirmative answer with a smart and explicit construction
of the domain. One year later, Boudabra and Markowsky published two
papers on the problem \cite{boudabra2019remarks,Boudabra2020}. In
the first paper, the authors showed that the problem is solvable for
any distribution with a finite $p^{th}$ moment whenevr $p<1$. That is, they extended
Gross's technique to cover all such distributions. Furthermore, a
uniqueness criterion was established. The second paper introduced
a new category of domains that solve the Skorokhod embedding problem
and yet provided a uniqueness criterion. The authors coined the term
$\textit{\ensuremath{\mu}-domain}$ to denote any simply connected
domain that solves the problem. The differences between the two solutions
are summarized as follows :
\begin{itemize}
\item Gross construction :
\begin{itemize}
\item[$\centerdot$]  $U$ is symmetric over the real line. 
\item[$\centerdot$]  $U$ is $\Delta$-convex, i.e the segment joining any point of $U$
and its symmetric one over the real axis remains inside $U$.
\item[$\centerdot$]  If $\mu(\{x\})>0$ then $\partial U$ contains a vertical line segment
(possibly infinite). 
\end{itemize}
\item Boudabra-Markowsky construction :
\begin{itemize}
\item[$\centerdot$]  $U$ is $\Delta^{\infty}$-convex, i.e the upward half ray at any
point of $U$ remains inside $U$.
\item[$\centerdot$]  If $\mu(\{x\})>0$ then $\partial U$ contains a downward half ray
at $x$. 
\end{itemize}
\end{itemize}
Note that if $(a,b)$ is not supported by $\mu$, i.e., $(a,b)$ is
null set for $\mu$ \footnote{~It means that the c.d.f of $\mu$ is constant on $(a,b)$.},
then any $\mu$-domain must contain the vertical strip $(a,b)\times(-\infty,+\infty)$.

\begin{figure}[H]
~~~~~~~~~~~~~~\includegraphics[width=6cm,height=6cm,keepaspectratio]{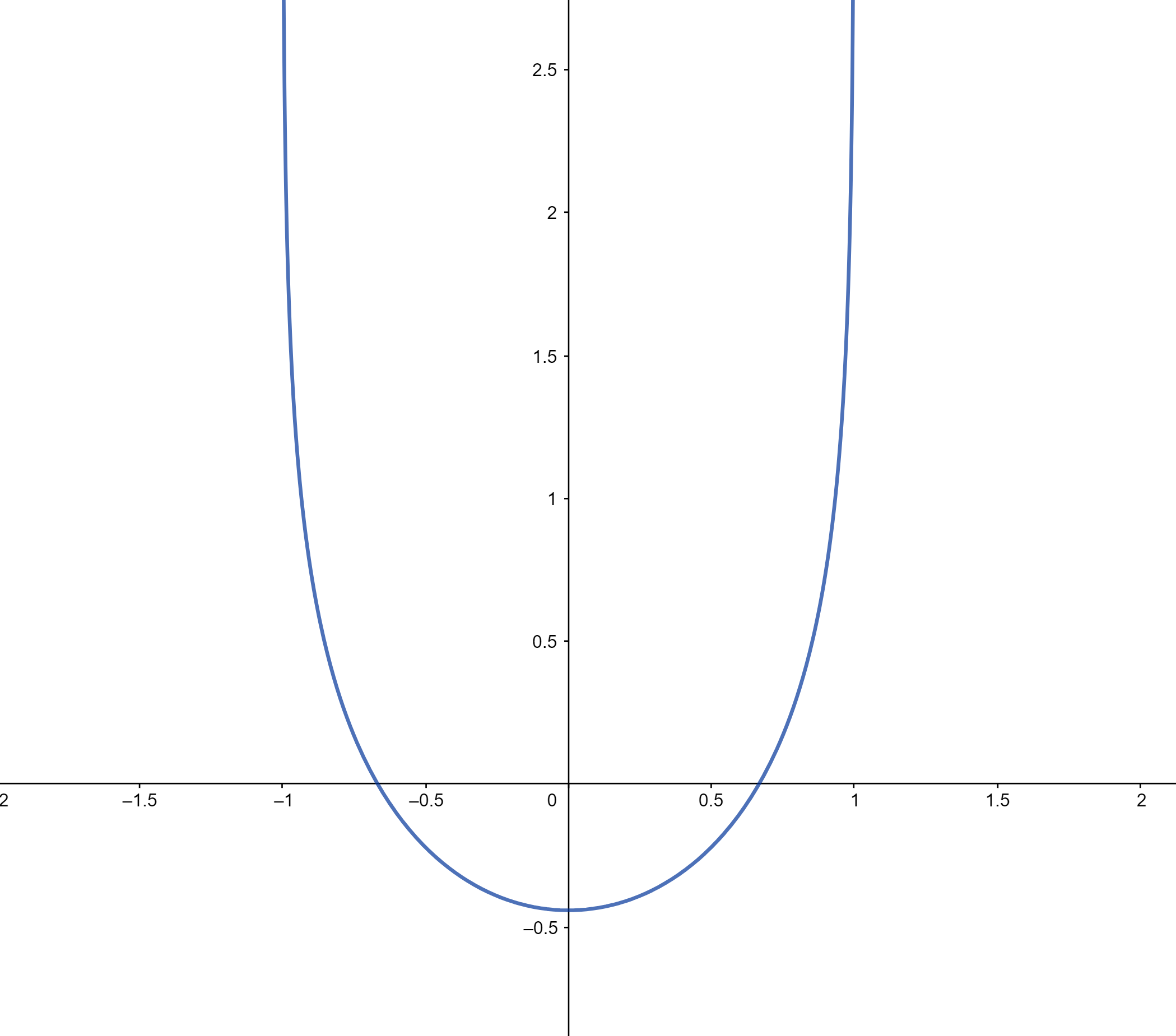}\includegraphics[width=5cm,height=5cm,keepaspectratio]{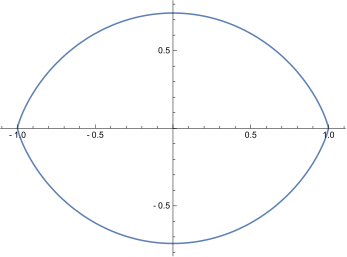}\caption{{\small{}For the uniform distribution on $(-1,1)$, the left domain
is Boudabra-Markowsky's solution while Gross's solution is on the right. }}

\end{figure}

Gross's clever approach was to construct a univalent function $G(z)={\displaystyle {\displaystyle {\textstyle \sum'}}}a_{n}z^{n}$
acting on the unit disc and then showing that $U:=G(\mathbb{D})$
solves the problem (the dashed sum means it starts from $n=1$). The
map $G$ is encoded by the quantile function of $\mu$. More precisely,
the coefficients $a_{n}$'s are the Fourier coefficients of the function
$\theta\in(-\pi,\pi)\mapsto q(\frac{\vert\theta\vert}{\pi})$ where
$q$ is the pseudo-inverse of the c.d.f of $\mu$. In particular 
\begin{equation}
q({\scriptstyle \frac{\vert\theta\vert}{\pi}})={\displaystyle {\displaystyle {\textstyle \sum'}}}a_{n}\cos(n\theta).\label{fouier equality}
\end{equation}
Note that the identity \ref{fouier equality} is meaningful both in
the almost everywhere sense and in the $L^{p}$ sense. So, Gross construction
is based on the knowledge of the Fourier coefficients of $q$, which
is often hard to extract knowing that the c.d.f itself may not be
explicit like in the case of normal distribution. Approximating solutions
is essential when exact ones are difficult to obtain, we thought to
investigate the question of approximating (in some sense) the underlying
$\mu$-domain by a certain sequence of domains. A natural idea is
to consider a sequence of probability distributions $\mu_{n}$ converging
to $\mu$ and track the behaviors of the sequence of the corresponding
$\mu_{n}$-domains. \\
\\

For the sake of completeness, we recall the basic definitions of c.d.f
and its underlying quantile function. The c.d.f of the distribution
$\mu$ is 
\[
F(x)=\mu((-\infty,x])
\]
 and its quantile function is 
\[
q(u)=\inf\{x\mid F(x)\geq u\}
\]
with $u\in(0,1)$. The quantile function plays an important role in
statistics. In fact, when fed with uniformly distributed inputs in
the interval $(0,1)$, $q$ generates data that samples as the underlying
probability distribution. In other words
\begin{equation}
q(\text{Uni}(0,1))\sim\mu.\label{sampling}
\end{equation}

Throughout the paper, $\mu$ denotes a probability distribution with
a finite $p$-th moment for some $p>1$, and $\mu_{n}$ represents
a sequence of probability measures. Additional assumptions will be
introduced whenever necessary. $\lambda$ stands for the Lebesgue
measure. 
\begin{defn}
We say that a sequence of probability distributions $\mu_{n}$ converges
weakly to $\mu$ if $F_{n}(u)=\mu_{n}((-\infty,x])$, c.d.f of $\mu_{n}$,
converges to $F(x)$ at every continuity point of $F$. 
\end{defn}

Weak convergence is weaker than the set-wise convergence. For further
details about the convergence of measures, the reader is invited to
have a look at \cite{billingsley2013convergence}. Our first result
is the following clue theorem. 
\begin{thm}
\label{clue theorem } The sequence $q_{n}$ converges almost everywhere
to $q$ on $(0,1)$ . 
\end{thm}

Before delving into the proof, we shall introduce the so-called the
strict pseudo inverse of $F$ defined by 
\[
q^{+}(u)=\inf\{x\mid F(x)>u\}
\]
with $u\in(0,1)$. Note that $q$ and $q^{+}$ are equal except at
values matching discontinuity points of $F$. Thus, the set $\{u\mid q(u)<q^{+}(u)\}$
is at most countable and hence
\[
q=q^{+}\,\,(a.e).
\]
The following properties can be obtained by just using the definitions
of $q$ and $q^{+}$:
\begin{enumerate}
\item $q^{+}(u)<x$ implies $u<F(x).$ 
\item $q(F(x))\leq x.$ 
\item $F(x)<u$ implies $x\leq q(u).$ 
\item $x<q(u)$ implies $F(x)<u.$ 
\end{enumerate}
\begin{proof}
Let $\mathscr{C}$ be the set of continuity of $F$. If $x\in\mathscr{C}$
then 
\begin{equation}
\lim_{n}F_{n}(x)=F(x).\label{lim F_n}
\end{equation}
Now we fix $\epsilon>0$ and $u\in(0,1)$, and consider $x\in\mathscr{C}\cap(q^{+}(u),q^{+}(u)+\epsilon)$.
In particular $u<F(x)$ by the first property mentioned earlier. (\ref{lim F_n})
implies that $u<F_{n}(x)$ for $n$ sufficiently large. Hence $q_{n}^{+}(u)\leq x$.
That is
\begin{equation}
\limsup_{n}q_{n}^{+}(u)<q^{+}(u)+\epsilon
\end{equation}
and so
\begin{equation}
\limsup_{n}q_{n}^{+}(u)\leq q^{+}(u)
\end{equation}
as $\epsilon$ is arbitrarily positive. Using the same approach, we
obtain 
\[
q^{+}(u)\leq\liminf_{n}q_{n}^{+}(u).
\]
As $q_{n}$ and $q_{n}^{+}$ are equal a.e, and so are their limits,
which ends the proof. 
\end{proof}
From a statistical perspective, Theorem \ref{clue theorem } establishes
that if a sequence of probability distributions $\mu_{n}$ converges
weakly to $\mu$, then one can construct a sequence of random variables
$\xi_{n}$ that converges almost surely to a random variable $\xi$.
Here, each $\xi_{n}$ is distributed according to $\mu_{n}$ and $\xi$
is distributed according to $\mu$. More precisely, this result is
achieved by defining 
\[
\xi_{n}=q_{n}(U),
\]
 where $U$ is a random variable uniformly distributed on the interval
$(0,1)$ (in virtue of \ref{sampling}) .

Theorem \ref{clue theorem } addresses the convergence of the sequence
of quantile functions. A related result can be found in \cite{barvinek1991convergence},
where the authors considered the following question: given the convergence
of a sequence of one-to-one maps $f_{n}$ to some function $f$, under
what conditions does $f_{n}^{-1}$ converge to $f^{-1}$? However,
there are two main differences between their result and our Theorem
\ref{clue theorem }. First, Theorem \ref{clue theorem } does not
require injectivity. Second, it only requires almost everywhere convergence
of the sequence $f_{n}$, whereas in \cite{barvinek1991convergence}
the authors assumed the uniform convergence of $f_{n}$ to a continuous
function $f$. 

In the rest of the paper, $\mu_{n}$ is assumed to converge weakly
to $\mu$. Following Gross' approach (described above), let $G_{n}$
(resp. $G$) be the underlying map generated from the distribution
$\mu_{n}$ (resp. $G$). That is 
\[
{\displaystyle G_{n}(z)={\displaystyle {\displaystyle {\textstyle \sum'}}}a_{k,n}z^{k}}
\]
\[
G(z)={\displaystyle {\displaystyle {\textstyle \sum'}}}a_{k}z^{k}
\]
with $a_{k,n}$ (resp. $a_{k}$) being the $k^{th}$ Fourier coefficient
of $q_{n}$ (resp. $q$). Now, we state our approximation theorem. 
\begin{thm}
\label{main theorem} If the sequence of quantile functions $q_{n}$
converges to $q$ in $L^{1}$ then the sequence $G_{n}$ converges
to $G$ uniformly on any compact subset of the unit disc. 
\end{thm}

\begin{proof}
Let $K$ be a compact subset of the unit disc and let $r\in(0,1)$
so that $K\subset r\mathbb{D}$. Hence 
\[
\begin{alignedat}{1}\vert G_{n}(z)-G(z)\vert & \leq{\displaystyle {\displaystyle {\textstyle \sum'}}}\vert a_{k}-a_{k,n}\vert r^{k}\\
 & \leq\sup_{k}\vert a_{k}-a_{k,n}\vert\frac{r}{1-r}\\
 & \leq\Vert q_{n}-q\Vert_{L^{1}}\frac{r}{1-r}.
\end{alignedat}
\]
 Therefore 
\[
\sup_{z\in K}\vert G_{n}(z)-G(z)\vert\leq\Vert q_{n}-q\Vert_{L^{1}}\frac{r}{1-r}\underset{n\rightarrow+\infty}{\longrightarrow}0.
\]
 Consequently 
\[
\sup_{z\in K}\vert G_{n}(z)-G(z)\vert\underset{n\rightarrow+\infty}{\longrightarrow}0
\]
which completes the proof.
\end{proof}
Geometrically, Theorem \ref{main theorem} can be interpreted as follows:
If $U_{n}$ is the Gross domain generated by $\mu_{n}$ then $U_{n}$
will converge, in terms of shape, to the Gross domain generated by
$\mu$. As one can see, the proof of the convergence of $G_{n}$ follows
immediately if we choose a sequence of distributions $\mu_{n}$ for
which the corresponding quantile functions $q_{n}$ converge in $L^{1}$
to $q$, that is, 
\[
q_{n}\overset{L^{1}}{\underset{n\rightarrow+\infty}{\longrightarrow}}q.
\]

Now, we give an example of $\mu$ and $\mu_{n}$ for which Theorem
\ref{main theorem} is satisfied. 
\begin{example}
\label{example} Consider the sequence $\xi_{n}=n\min(U_{1},...,U_{n})$
where $U_{i}$ are i.i.d of law $\text{Uni}(0,1)$. The measure $\mu_{n}$
of $\xi_{n}$ converges weakly to the measure of $\mathscr{E}xp(1)$.
In particular, the underlying quantile sequence of $\mu_{n}$ , given
by
\[
q_{n}(u)=n(1-(1-u)^{1/n}),
\]
 does converge to $q(u)=-\ln(1-u)$ a.e. Furthermore, one can check
that
\[
n(1-(1-u)^{1/n})\overset{L^{1}}{\underset{n\rightarrow+\infty}{\longrightarrow}}-\ln(1-u).
\]
\end{example}

A typical case when we can make Theorem \ref{main theorem} work is
when the quantile sequence is bounded. 
\begin{thm}
\label{bounded quantile} If the quantile sequence is bounded then
Theorem \ref{main theorem} holds.
\end{thm}

\begin{proof}
Notice that if $q_{n}$ is bounded then $q$ is bounded. Hence, by
combining the dominated convergence theorem and Theorem \ref{clue theorem },
we obtain 
\[
\Vert q_{n}-q\Vert_{L^{1}}\underset{n\rightarrow+\infty}{\longrightarrow}0
\]
which ends the proof. 
\end{proof}
Theorem \ref{bounded quantile} is quite stringent in the absence
of the support provided by Theorem \ref{clue theorem }, especially
when dealing with unbounded distributions, i.e with unbounded supports.
Our goal is then to identify and impose the minimal additional conditions
on the sequence $q_{n}$ that are sufficient to ensure its $L^{1}$
convergence to $q$. By doing so, we can extend the applicability
of our convergence result (Theorem \ref{main theorem}) to a broader
class of distributions, namely those with unbounded support. We shall
keep the usage of $q_{n}$ and $q$ as before, i.e $q_{n}$ is the
quantile of $\mu_{n}$ and $q$ is the quantile of $\mu$. 

In a recent paper \cite{alves2024mode}, the authors introduced a
novel notion of convergence called $\alpha_{p}$-convergence. For
the sake of completeness, we briefly recall its definition and some
relevant properties. 
\begin{defn}
A sequence $f_{n}$ of real valued measurable functions is said to
be $\alpha_{p}$-convergent to a measurable function $f$ if there
exists a sequence of measurable sets $B_{n}$ of expanding sizes,
i.e $\lambda(B_{n}^{c})\underset{n}{\rightarrow}0$, such that 
\begin{equation}
\int_{B_{n}}|f_{n}-f|^{p}d\lambda\underset{n\rightarrow+\infty}{\longrightarrow}0.
\end{equation}
\end{defn}

It is obvious that the $\alpha_{p}$-convergence is weaker than the
$L^{p}$ convergence. Similarly, if the space is of finite Lebesgue
measure then the $\alpha_{p}$-convergence implies the $\alpha_{q}$-convergence
whenever $0\leq q<p$. A handy equivalence relating both modes of
convergence is the following. 
\begin{prop}
\label{L^p and alpha p} \cite{alves2024mode} The $L^{p}$ convergence
of a sequence of measurable functions $f_{n}$ to some $f$ is equivalent
to its $\alpha_{p}$-convergence to $f$ combined with the condition 
\begin{equation}
\forall E_{n},\lambda(E_{n})\underset{n}{\rightarrow}0\Longrightarrow\int_{E_{n}}|f_{n}-f|d\lambda\underset{n\rightarrow+\infty}{\longrightarrow}0.\label{extra condition}
\end{equation}
\end{prop}

When the ambient space is of finite measure, typically bounded intervals,
then the next result holds.
\begin{thm}
\label{alpha p and measure} \cite{10.14321/realanalexch.49.2.1720434606}
If the space is of finite measure then the $\alpha_{p}$-convergence
and the convergence in measure are equivalent. 
\end{thm}

Note that the finiteness of the ambient space is not required to show
that the $\alpha_{p}$-convergence implies that in measure. As the
sequence $q_{n}$ converges to $q$ a.e then the convergence occurs
yet in measure, i.e for all $\delta>0$, we have 
\begin{equation}
\lambda\{|q_{n}-q|>\delta\}\underset{n\rightarrow+\infty}{\longrightarrow}0.
\end{equation}
Therefore by Theorem \ref{alpha p and measure}, $q_{n}$ $\alpha_{p}$-converges
to $q$. In order to fulfill the $L^{1}$ convergence of $q_{n}$
to $q$, we still need to check \ref{extra condition}. However, such
condition is a bit unpractical to satisfy as one needs to test all
measurable sets $E_{n}$ of shrinking size. Using the fact that $q_{n}$
is non decreasing, we can weaken that assumption to apply it only
to sets near the endpoints of the interval $(0,1)$. More precisely
\begin{prop}
\label{weak} If 
\begin{equation}
\max(\int_{0}^{\delta}\vert q_{n}-q\vert du,\int_{1-\delta}^{1}\vert q_{n}-q\vert du)\underset{n\rightarrow+\infty}{\longrightarrow}0\label{delta integral}
\end{equation}
for some positive $0<\delta<\frac{1}{2}$, then 
\begin{equation}
\Vert q_{n}-q\Vert_{1}\underset{n\rightarrow+\infty}{\longrightarrow}0.\label{propo7}
\end{equation}
\end{prop}

\begin{proof}
Let $S$ a set of $(0,1)$ where $q_{n}$ converges to $q$. In particular,
$\lambda(S)=1$. Set $S_{\delta}=[\delta,1-\delta]\cap S$.
As $q_{n}$ is non-decreasing and converges a.e to $q$, then for
any $\epsilon>0,u\in S_{\delta}$, there is $N$ such that 
\[
\underset{u\in S}{\inf}q(u)-\epsilon\leq q_{n}(u)\leq\underset{u\in S}{\sup}q(u)+\epsilon
\]
whenever $n\geq N$. Hence, $q_{n}$ gets bounded on $S_{\delta}$.
The dominated convergence theorem infers that 
\begin{equation}
\int_{S_{\delta}}\vert q_{n}-q\vert d\lambda=\int_{\delta}^{1-\delta}\vert q_{n}-q\vert d\lambda\underset{n\rightarrow+\infty}{\longrightarrow}0,\label{S_delta}
\end{equation}
which completes the proof as (\ref{delta integral}) and (\ref{S_delta})
lead to (\ref{propo7}). 
\end{proof}
Consequently, we obtain a generalized version of theorem \ref{main theorem}.
This would cover probability distributions of unbounded supports; 
in particular, their underlying quantile functions are unbounded. 
\begin{thm}
\label{thm:Assume-the-sequence} Assume the sequence $q_{n}$ satisfies
the condition in Proposition \ref{weak} . Then $G_{n}$ converges
to $G$ compactly uniformly. 
\end{thm}

Truncation is a fundamental technique for constructing a sequence
of measures $\mu_{n}$ that converges weakly to $\mu$, while ensuring
that 
\[
q_{n}\overset{L^{1}}{\underset{n\rightarrow+\infty}{\longrightarrow}}q.
\]
If $X:((0,1),\lambda)\mapsto\mathbb{R}$ is a random variable distributed
according to $\mu$, then defining the truncated sequence $X_{n}=X\mathbf{1}_{\{|X|\leq n\}}$
leads to a corresponding sequence of probability measures. Using symmetrization
theory \cite{kesavan2006symmetrization}, we have 
\[
\ensuremath{\|q-q_{n}\|_{1}\leq\|X-X_{n}\|_{1}\to0}.
\]

Hence, this sequence, $X_{n}$, generates quantiles satisfying the
desired property. For the sake of completeness, we provide the closed
form of $\mu_{n}$ and $q_{n}$ of $X_{n}=X\mathbf{1}_{\{|X|\leq n\}}$.
\begin{prop}
Let $\mu$ be an unbounded probability measure with some finite $p^{th}$
moment, $p>1$. Then the sequence of measures defined by 
\[
d\mu_{n}(\{x\})=\left[((F(-n)+1-F(n))+\mu(\{0\}))\delta_{0}(x)+d\mu(\{x\})\right]1_{x\in[-n,n]}
\]
converges weakly to $\mu$. Furthermore, the underlying quantile function
$q_{n}$ converges to $q$ in $L^{1}$.
\end{prop}

\begin{proof}
The c.d.f of $\mu_{n}$ is 
\[
F_{n}(x)=\begin{cases}
F(x)-F(-n^{-}) & -n\leq x<0\\
F(x)+1-F(n) & 0\leq x\leq n.
\end{cases}
\]
Hence 
\[
q_{n}(u)=\begin{cases}
q(u+F(-n^{-})) & (0,F(0^{-})-F(-n^{-}))\\
0 & u\in(F(0^{-})-F(-n^{-}),F(0)+1-F(n))\\
q(u+F(n)-1) & u\in(F(0)+1-F(n),1).
\end{cases}
\]
The truncation technique applied to the exponential distribution generates
a sequence $q_{n}$ different from the one in \ref{example}. 
\end{proof}
\begin{example}
Let $\xi$ be an exponentially distributed random variable, say with
parameter $1$, and let $\mu_{n}$ be the sequence of measures given
by 
\[
d\mu_{n}(x)=e^{-n}\delta_{0}(x)+e^{-x}1_{\{0<x\leq n\}}.
\]
The corresponding quantile functions $q_{n}$ are given by 
\[
q_{n}(u)=-\ln(1-u)1_{\{e^{-n}\leq u\}}.
\]
One can argue that this example does not comply with the centering
of random variables required in Gross's method, but this requirement
is achieved by a simple shift.  
\end{example}

\section{Implementation }

In this section, we shall provide the implementation framework of
the results established in the first section. For instance, $\mu$
stands for a probability measure of a bounded connected support $(a,b)$,
possibly with a finite number of atoms $a_{1}<...<a_{s}$. The quantile
function of $\mu$ is 
\begin{equation}
q(u)=\sum_{i=1}^{s+1}F^{-1}(u)1_{\{u\in(F(a_{k-1}),F(a_{k}^{-}))\}}+\sum_{k=1}^{s}a_{k}1_{\{u\in(F(a_{k}^{-}),F(a_{k}))\}}\label{q(u)}
\end{equation}
with the conventions $a_{0}=a_{0}$ and $a_{s+1}=b$ and $\sum_{\emptyset}=0$.
The first term in the R.H.S of (\ref{q(u)}) represents the continuous
part of the distribution $\mu$ while the second term encodes the
weights of the atoms (the discrete part). Now, we provide a sequence
of probability measures $\mu_{n}$ converging to $\mu$ in distribution.
But first, and in order to avoid congestion of symbols, we shall use
the following notations :
\begin{itemize}
\item[$\centerdot$]  $\mathbf{A}=\{a_{1},...,a_{s}\}$. 
\item[$\centerdot$]  $A_{i,s}=\begin{cases}
(a,a_{1}) & i=0\\
(a_{i},a_{i+1}) & i=1,...,s-1\\
(a_{s},b) & i=s
\end{cases}$. \\
If $a=a_{1}$ (resp. $a_{s}=b$) we take $A_{0,s}=\emptyset$ (resp.
$A_{s,s}=\emptyset$)
\item[$\centerdot$]  $p_{i}=\mu(\{a_{i}\})$, $i=1,...,s$. 
\item[$\centerdot$]  $x_{k}=a+(b-a)\frac{k}{n}$, $k=0,...,n$.
\end{itemize}
\begin{thm}
The sequence of probability measures $\mu_{n}$ given by 
\begin{equation}
\mu_{n}=\sum_{\underset{x_{k-1},x_{k}\notin\mathbf{A}}{k=1,...,n}}^{n}(F(x_{k})-F(x_{k-1}))\delta_{x_{k}}+\sum_{k=1}^{s}p_{k}\delta_{a_{k}}\label{=00005Cmu_n}
\end{equation}
converges weakly to $\mu$. 
\end{thm}

\begin{proof}
At the atoms of $\mu$, notice that $\mu_{n}(\{a_{i}\})=p_{i}=\mu(\{a_{i}\})$
for $i=1,...,s$. Now let $x$ be a value different of the atoms,
i.e a value where the c.d.f $F$ of $\mu$ is continuous. We have
\[
\begin{alignedat}{1}\mu_{n}((-\infty,x]) & =\sum_{\underset{x_{k-1},x_{k}\notin\mathbf{A}}{1\leq k\leq n}}^{n}(F(x_{k})-F(x_{k-1}))1_{\{x\leq x_{k}\}}+\sum_{k=1}^{s}p_{k}1_{\{a_{k}\leq x\}}\\
 & =\sum_{\underset{x_{k-1},x_{k}\notin\mathbf{A},x_{k}\leq x}{1\leq k\leq n}}^{n}(F(x_{k})-F(x_{k-1})+\sum_{k=1}^{s}p_{k}1_{\{a_{k}\leq x\}}\\
 & =\sum_{1\leq k\leq\lfloor n\frac{x-a}{b-a}\rfloor}^{n}(F(x_{k})-F(x_{k-1}))+\sum_{k=1}^{s}p_{k}1_{\{a_{k}\leq x\}}\\
 & -\sum_{\underset{x_{k-1}\text{ or }x_{k}\in\mathbf{A},x_{k}\leq x}{k=1,...,n}}^{n}(F(x_{k})-F(x_{k-1})\\
 & =F(a+{\textstyle \lfloor n\frac{x-a}{b-a}\rfloor\frac{b-a}{n}})-F(a)\\
 & +\sum_{k=1}^{s}p_{k}1_{\{a_{k}\leq x\}}-\sum_{\underset{x_{k-1}\text{ or }x_{k}\in\mathbf{A},x_{k}\leq x}{k=1,...,n}}^{n}(F(x_{k})-F(x_{k-1})\\
 & \underset{n\rightarrow+\infty}{\longrightarrow}F(a+x-a)\\
 & =F(x).
\end{alignedat}
\]
The zero limit of 
\[
\sum_{k=1}^{s}p_{k}1_{\{a_{k}\leq x\}}-\sum_{\underset{x_{k-1}\text{ or }x_{k}\in\mathbf{A},x_{k}\leq x}{k=1,...,n}}^{n}(F(x_{k})-F(x_{k-1})
\]
 follows from the facts that $F(a_{i}^{+})=F(a_{i})$ and $F(a_{i}^{-})=F(a_{i})-p_{i}$.
\end{proof}
\begin{rem}
The reason to track only points $x_{k-1},x_{k}\not\in A$ is to make
$\mu_{n}$ coincide with $\mu$ on the set $A$. This would prevents
us from extra effort to show convergence on $A$. 
\end{rem}

The underlying quantile function of the sequence $\mu_{n}$ (defined
by (\ref{=00005Cmu_n})) is obtainable using its definition in the
first section. We have

\[
q_{n}(u)=\sum_{\underset{x_{k-1},x_{k}\notin\mathbf{A}}{k=1,...,n}}^{n}x_{k}1_{\{u\in(F(x_{k-1}),F(x_{k}))\}}+\sum_{k=1}^{s}a_{k}1_{\{u\in(F(a_{k}^{-}),F(a_{k}))\}}.
\]
As $q_{n}$ is bounded then it converges to $q$ in $L^{1}$ in virtue
of Theorem \ref{bounded quantile}. Our next result gives the rate
at which this convergence occurs.  
\begin{prop}
We have 
\begin{equation}
\Vert q-q_{n}\Vert_{1}\leq\frac{b-a}{n}+\varpi_{n,A}\label{estimate}
\end{equation}
where $\varpi_{n,A}$ is a quantity that goes to zero as $n$ gets
large. 
\end{prop}

\begin{proof}
Recall that $q_{n}$ and $q$ coincide on the set $A$. So 
\[
\begin{alignedat}{1}\Vert q-q_{n}\Vert & _{1}=\Vert\sum_{\underset{x_{k-1},x_{k}\notin\mathbf{A}}{k=1,...,n}}^{n}x_{k}1_{\{u\in(F(x_{k-1}),F(x_{k}))\}}-\sum_{i=1}^{s+1}F^{-1}(u)1_{\{u\in(F(a_{i-1}),F(a_{i}^{-}))\}}\Vert_{1}\\
 & =\Vert\sum_{i=1}^{s+1}\left(\sum_{\underset{x_{k-1},x_{k}\in A_{i,s}}{k=1,...,n}}^{n}x_{k}1_{\{u\in(F(x_{k-1}),F(x_{k}))\}}-F^{-1}(u)1_{\{u\in(F(a_{i-1}),F(a_{i}^{-}))\}}\right)\Vert_{1}\\
 & \leq\sum_{i=1}^{s+1}\left(\sum_{\underset{x_{k-1},x_{k}\in A_{i,s}}{k=1,...,n}}^{n}\Vert(x_{k}-F^{-1}(u))1_{\{u\in(F(x_{k-1}),F(x_{k}))\}}\Vert_{1}\right)\\
 & +\underset{=\varpi_{n,A}}{\underbrace{\sum_{i=1}^{s+1}\left(a_{i}(F(a_{i}^{-})-F(a_{i}-{\textstyle \frac{b-a}{n}}))+(a_{i-1}+{\textstyle \frac{b-a}{n}})(F(a_{i-1}+{\textstyle \frac{b-a}{n}})-F(a_{i-1}))\right)}}\\
 & \leq\frac{b-a}{n}\sum_{i=1}^{s+1}\left(\sum_{\underset{x_{k-1},x_{k}\in A_{i,s}}{k=1,...,n}}^{n}(F(x_{k})-F(x_{k-1}))\right)+\varpi_{n,A}\\
 & \leq\frac{b-a}{n}\sum_{i=1}^{s+1}\left((F(a_{i-1})-F(a_{i}))\right)+\varpi_{n,A}\\
 & \leq\frac{b-a}{n}+\varpi_{n,A}.
\end{alignedat}
\]
As 
\[
F(a_{i}^{-})-F(a_{i}-{\textstyle \frac{b-a}{n}}),(F(a_{i-1}+{\textstyle \frac{b-a}{n}})-F(a_{i-1}))\underset{n\rightarrow+\infty}{\longrightarrow}0
\]
by the càd-làg feature of the c.d.f $F$, the proof is finished. 
\end{proof}
It is worth scrutinizing the term $\varpi_{n,A}$. According to the
proof, the quantity $\varpi_{n,A}$ encodes the contribution of the
atoms of $\mu$ on the convergence. If $\mu$ has no atoms then $A=\emptyset$,
and it follows that $\varpi_{n,A}=0$.
\begin{cor}
When $\mu$ has no atoms then 
\[
\Vert q-q_{n}\Vert_{1}\leq\frac{b-a}{n}.
\]
\end{cor}

When $F$ has a bounded p.d.f $F'$ on $(a,b)-A$ then we can estimate
the decay of the quantity $\varpi_{n,A}$ up to $\frac{1}{n^{2}}$
factor. That is
\begin{prop}
If $F$ has a bounded p.d.f on $(a,b)-A$ then 
\[
\Vert q-q_{n}\Vert_{1}\leq\frac{\alpha}{n}+\frac{\beta}{n^{2}}
\]
for some positive constants $\alpha,\beta$. 
\end{prop}

\begin{proof}
We have 
\begin{equation}
\begin{alignedat}{1}\varpi_{n,A} & =\sum_{i=1}^{s+1}\left(a_{i}(F(a_{i}^{-})-F(a_{i}-{\textstyle \frac{b-a}{n}}))+(a_{i-1}+{\textstyle \frac{b-a}{n}})(F(a_{i-1}+{\textstyle \frac{b-a}{n}})-F(a_{i-1}))\right)\\
 & \leq\sum_{i=1}^{s+1}\left(a_{i}{\textstyle \frac{b-a}{n}}\Vert F_{\mid A_{i,s}}'\Vert_{\infty}))+(a_{i-1}+{\textstyle \frac{b-a}{n}}){\textstyle \frac{b-a}{n}}\Vert F_{\mid A_{i,s}}'\Vert_{\infty}))\right)\\
 & ={\textstyle \frac{b-a}{n}}\sum_{i=1}^{s+1}\Vert F_{\mid A_{i,s}}'\Vert_{\infty}(a_{i}+a_{i-1}+{\textstyle \frac{b-a}{n}})\\
 & =\left(\sum_{i=1}^{s+1}\Vert F_{\mid A_{i,s}}'\Vert_{\infty}(a_{i}+a_{i-1})\right){\textstyle \frac{b-a}{n}}+\left(\sum_{i=1}^{s+1}\Vert F_{\mid A_{i,s}}'\Vert_{\infty}\right)\left({\textstyle \frac{b-a}{n}}\right)^{2}.
\end{alignedat}
\label{varomega}
\end{equation}
Therefore, by combining (\ref{varomega}) and (\ref{estimate}), we
obtain 
\[
\Vert q_{n}-q\Vert_{1}\leq\frac{\alpha}{n}+\frac{\beta}{n^{2}}
\]
where 
\[
\alpha=(b-a)(1+\sum_{i=1}^{s+1}\Vert F_{\mid A_{i,s}}'\Vert_{\infty}(a_{i}+a_{i-1}))
\]
and 
\[
\beta=(b-a)^{2}\sum_{i=1}^{s+1}\Vert F_{\mid A_{i,n}}'\Vert_{\infty}.
\]
\end{proof}
When the distribution $\mu$ has a $\mathcal{C}^{2}$ c.d.f of known
p.d.f $f$. Then, using Taylor expansion, the variation $F(x_{k})-F(x_{k-1})$
can be replaced by the quantity 
\[
\frac{f(x_{k-1})}{n}(b-a)+o(\frac{1}{n}).
\]
 In particular we can use the sequence of measures
\begin{equation}
\mu_{n}=\sum_{\underset{x_{k-1},x_{k}\notin\mathbf{A}}{k=1,...,n}}^{n}\frac{f(x_{k-1})}{n}(b-a)\delta_{x_{k}}+\sum_{k=1}^{s}p_{k}\delta_{a_{k}}\label{new =00005Cmu}
\end{equation}
instead of the one defined in \ref{=00005Cmu_n}. The sequence defined
in \ref{new =00005Cmu} corresponds to the following $q_{n}$ : 
\[
q_{n}(u)=\sum_{\underset{x_{k-1},x_{k}\notin\mathbf{A}}{k=1,...,n}}^{n}x_{k}1_{\{u\in(\sigma_{k-1},\sigma_{k})\}}+\sum_{k=1}^{s}a_{k}1_{\{u\in(F(a_{k}^{-}),F(a_{k}))\}}
\]
with 
\begin{equation}
\sigma_{k}=\frac{b-a}{n}(f(a)+\cdots+f(x_{k-1}))\label{sigma_k}
\end{equation}
for $k\geq0$.
\begin{rem}
Note that the sequence $(\frac{f(x_{k-1})}{n}(b-a))_{k}$ does not
necessarily sum to one even, but this does not present an issue for
the results established about convergence.
\end{rem}

In order to get the sequence of $\mu_{n}$-domains, it is enough to
know the parametrization of the boundary of the $\mu$-domains obtained
by Gross's technique. Such parametrization is given by 
\[
t\in(-1,1)\longmapsto(q(\vert t\vert),H\{q(\vert\cdot\vert)\}(t))
\]
 (\cite{Boudabra2020,gross2019}), where $H$ is the Hilbert transform. 
\begin{defn}
The Hilbert transform of a $2\pi$- periodic function $f$ is defined
by 
\[
H\{f\}(x):=PV\left\{ \frac{1}{2\pi}\int_{-\pi}^{\pi}f(x-t)\cot(\frac{t}{2})dt\right\} =\lim_{\eta\rightarrow0}\frac{1}{2\pi}\int_{\eta\leq|t|\leq\pi}f(x-t)\cot(\frac{t}{2})dt
\]
where $PV$ denotes the Cauchy principal value, which is required
here as the trigonometric function $t\longmapsto\cot(\cdot)$ has
poles at $k\pi$ with $k\in\mathbb{Z}$. The Hilbert transform does
exist for any function in $L_{2\pi}^{p}$ with $p\geq1$. However,
$H$ is a bounded operator only when $p>1$ \cite{butzer1971hilbert,king2009hilbert}
.
\end{defn}

As our sequence of measures encode step functions, we provide the
calculation of the Hilbert transform of $1_{\{a<\vert x\vert<b\}}$:
\begin{equation}
H\{1_{\{a<\vert x\vert<b\}}\})(u)=-\frac{1}{\pi}\left(\ln(\text{\ensuremath{\frac{\sin(\frac{u-b}{2})\sin(\frac{u+a}{2})}{\sin(\frac{u-a}{2})\sin(\frac{u+b}{2})}}})\right)=\frac{1}{\pi}\ln\left(\frac{\sin(\frac{u-a}{2})\sin(\frac{u+b}{2})}{\sin(\frac{u-b}{2})\sin(\frac{u+a}{2})}\right).\label{hilbert}
\end{equation}

Using these formulas, we get the Hilbert transform of $q_{n}(\vert\cdot\vert)$.
Recall the expression of $q_{n}$ : 
\[
\sum_{\underset{x_{k-1},x_{k}\notin\mathbf{A}}{k=1,...,n}}^{n}x_{k}1_{\{u\in(F(x_{k-1}),F(x_{k}))\}}+\sum_{k=1}^{s}a_{k}1_{\{u\in(F(a_{k}^{-}),F(a_{k}))\}}.
\]
Therefore
\[
\begin{alignedat}{1}H\{q_{n}(\vert\cdot\vert)\} & =\sum_{\underset{x_{k-1},x_{k}\notin\mathbf{A}}{k=1,...,n}}^{n}x_{k}H\{1_{\{\vert u\vert\in(F(x_{k-1}),F(x_{k}))\}}\}+\sum_{k=1}^{s}a_{k}H\{1_{\{\vert u\vert\in(F(a_{k}^{-}),F(a_{k}))\}}\}\\
 & =\frac{1}{\pi}\sum_{\underset{x_{k-1},x_{k}\notin\mathbf{A}}{k=1,...,n}}^{n}x_{k}\ln\left(\frac{\sin(\frac{u-F(x_{k-1})}{2})\sin(\frac{u+F(x_{k})}{2})}{\sin(\frac{u-F(x_{k})}{2})\sin(\frac{u+F(x_{k-1})}{2})}\right)\\
 & +\frac{1}{\pi}\sum_{k=1}^{s}a_{k}\ln\left(\frac{\sin(\frac{u-(F(a_{k}^{-})}{2})\sin(\frac{u+F(a_{k})}{2})}{\sin(\frac{u-F(a_{k})}{2})\sin(\frac{u+(F(a_{k}^{-})}{2})}\right).
\end{alignedat}
\]

\section{Examples}

In this concluding section, we present numerical simulations of various
distributions to demonstrate the convergence of our scheme to the
$\mu$-domains. For ease of explanation, we employ the following scheme.
: For a continuous probability distribution $\mu$ of a p.d.f $f$,
we shall use  
\[
q_{n}(u)=\sum_{k=1}^{n}x_{k}1_{\{u\in(\sigma_{k-1},\sigma_{k})\}}
\]
with $\sigma_{k}$ as defined in \ref{sigma_k}.

\begin{figure}[H]
\begin{raggedright}
\includegraphics[width=15cm,totalheight=37cm,keepaspectratio]{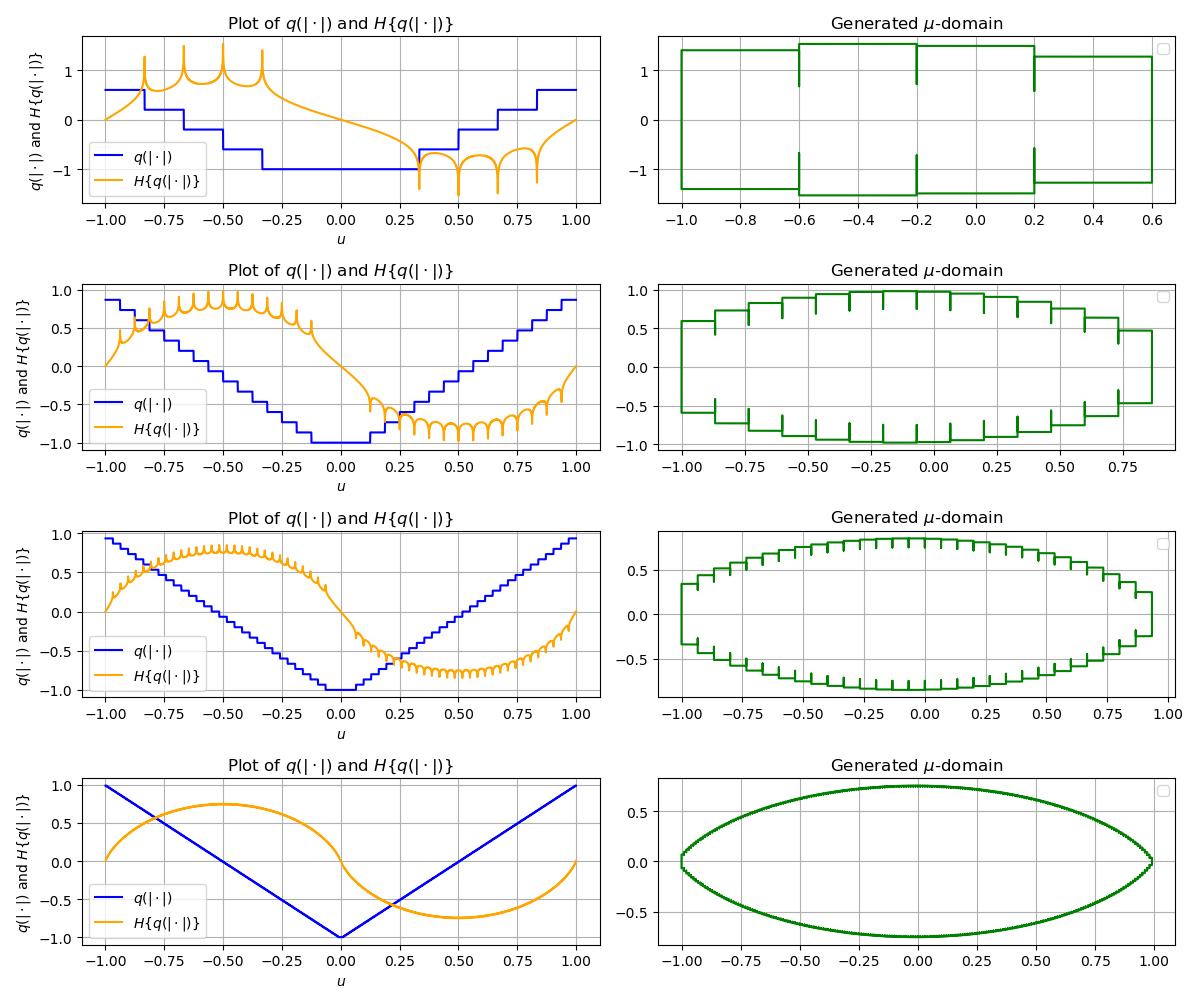}\caption{Approximation of the $\mu$-domain generated from $\mu=\text{Uni}(-1,1)$
for $n=5,15,30,200$ (top to bottom)}
\par\end{raggedright}
\end{figure}

\begin{figure}[H]
\raggedright{}\includegraphics[width=15cm,totalheight=30cm,keepaspectratio]{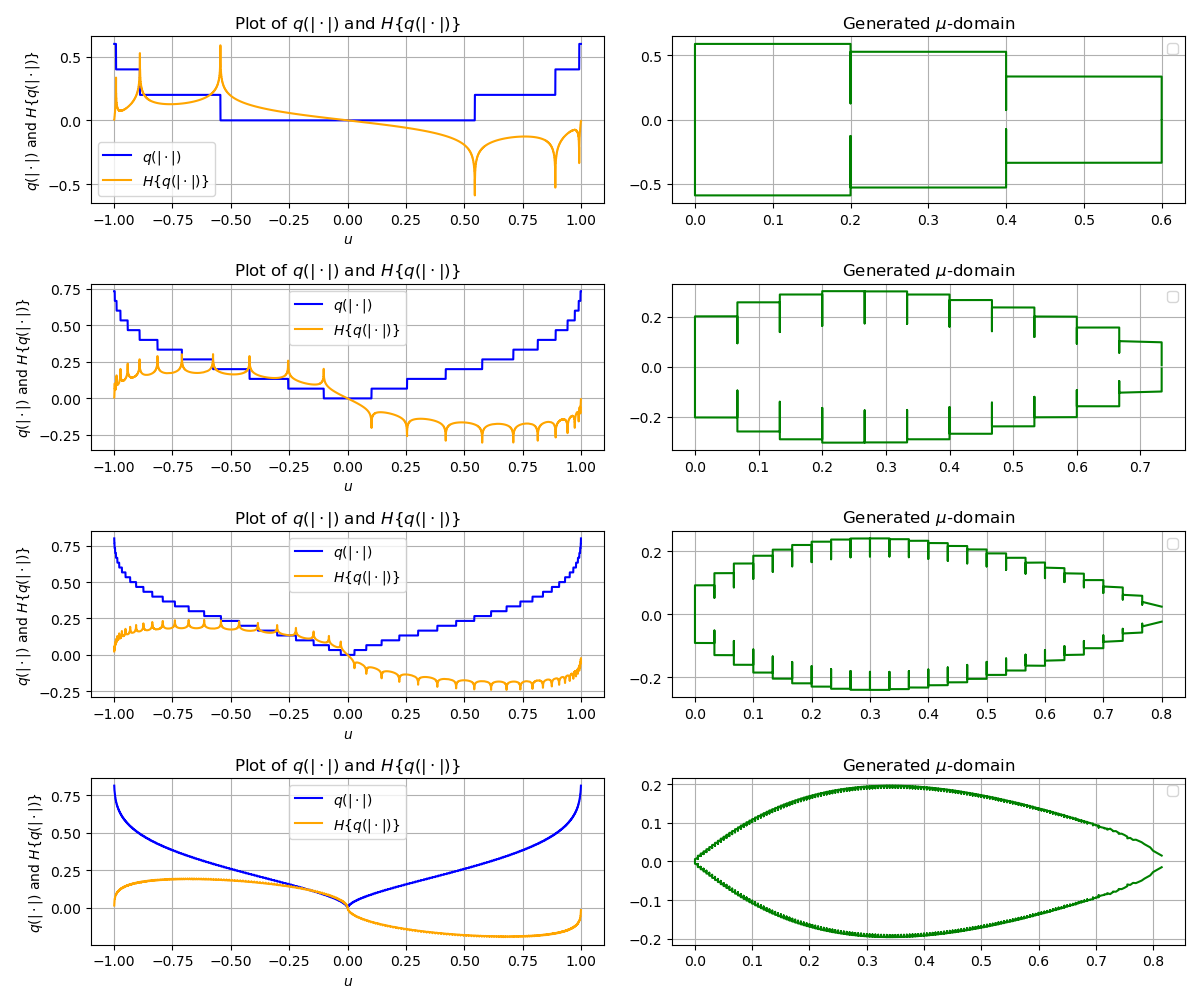}\caption{Approximation of the $\mu$-domain generated from $\mu=\text{Beta}(2,5)$
for $n=5,15,30,200$ (top to bottom)}
\end{figure}

\begin{figure}[H]
\raggedright{}\includegraphics[width=15cm,totalheight=30cm,keepaspectratio]{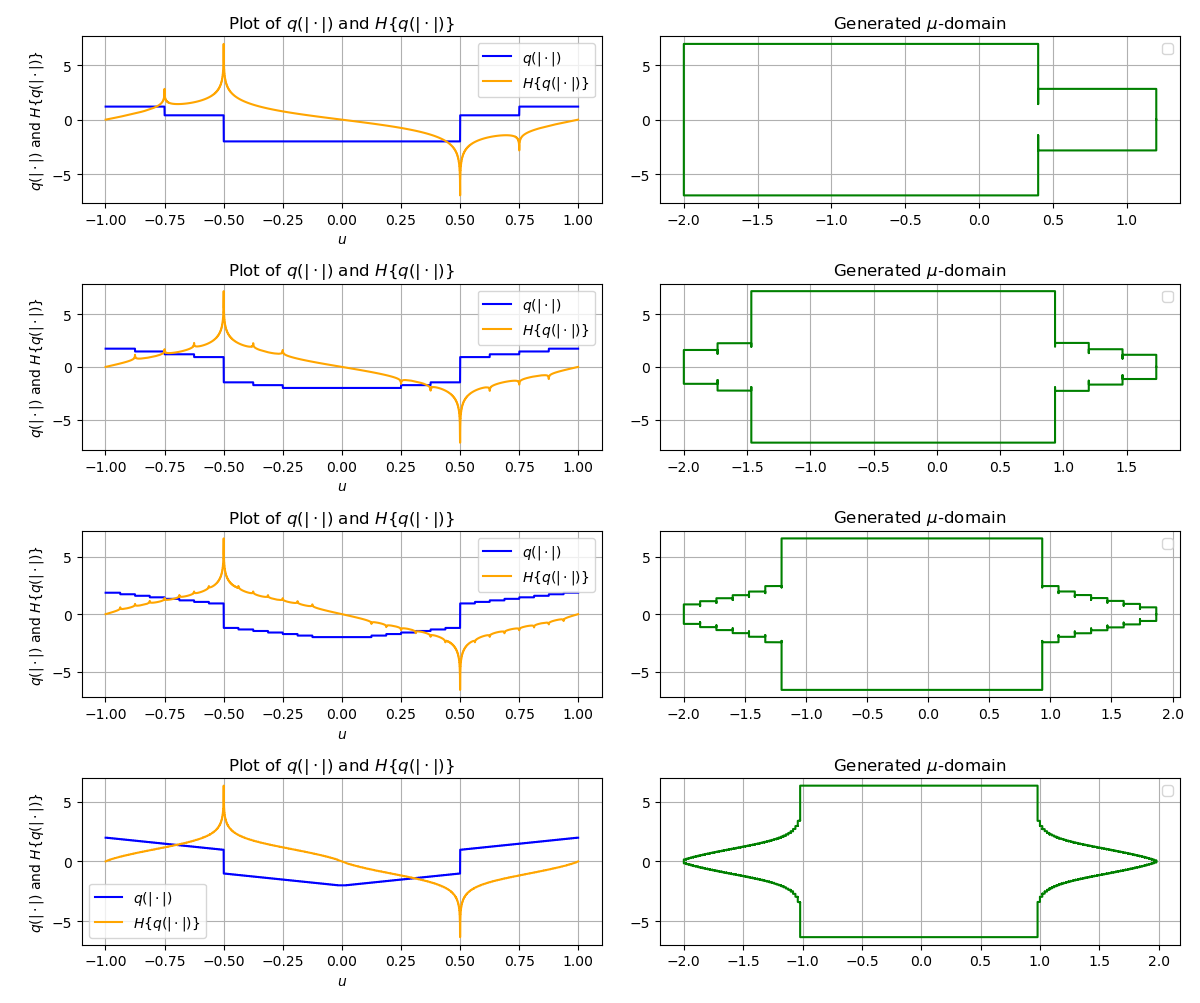}\caption{Approximation of the $\mu$-domain generated from $\mu=\text{Uni}((-2,-1)\cup(1,2))$
for $n=5,15,30,200$ (top to bottom). The $\mu$-domain contains the
vertical strip $\{-1<x<1\}$. The green horizontal segments appear
due to the discrete step-wise nature of the calculation used in the
scheme. }
\end{figure}

\begin{figure}[H]
\raggedright{}\includegraphics[width=15cm,totalheight=30cm,keepaspectratio]{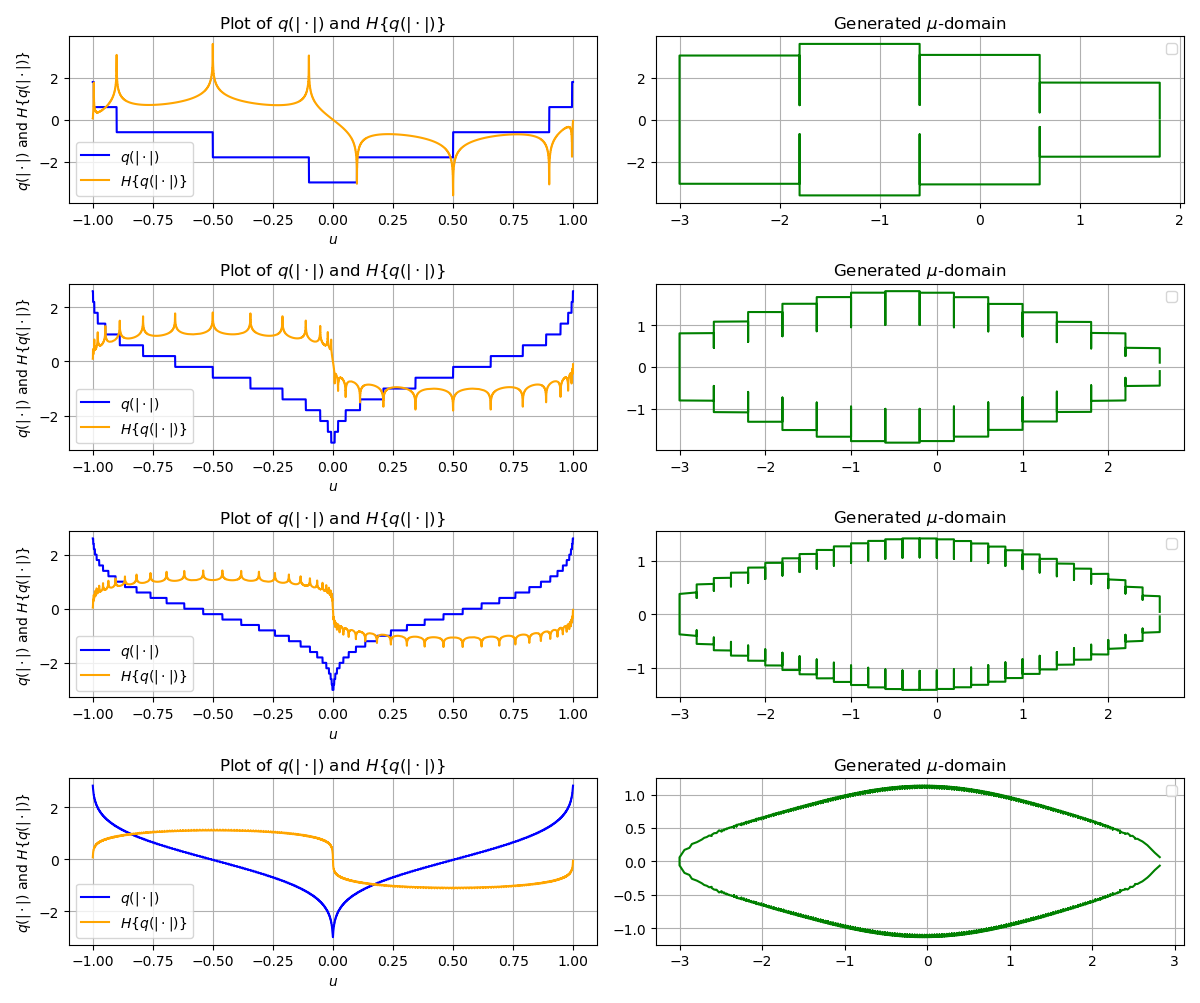}\caption{Approximation of the $\mu$-domain generated from $\mu=\mathcal{N}(0,1)_{\mid(-3,3)}$
(truncated normal distribution) for $n=5,15,30,200$ (top to bottom). }
\end{figure}

When it comes to unbounded distributions or distributions with immense
supports, applying directly the aforementioned scheme is not practical
as the step-size $\frac{b-a}{n}$ might not be small. However, the
following scaling property that can be obtained easily.
\begin{lem}
Let $\eta$ and $\widetilde{\eta}$ be two probability measures of
quantile functions $q$ and $\widetilde{q}$. If $\widetilde{q}=\alpha q+\beta$
for some $\alpha,\beta\in\mathbb{R}$ then the $\mu$-domains $U$
and $\widetilde{U}$ are related by 
\[
\widetilde{U}=\alpha U+\beta
\]
\end{lem}

In practice, if $\eta$ is centered and has a ``large'' support
$(a,b)$ with $a<0<b$ then 
\[
\widetilde{q}=\frac{1}{b-a}q-a
\]
 will be supported in the range $(0,1)$. Therefore, we apply our
scheme to $\widetilde{q}$ and construct $\widetilde{U}$. Then the
$\mu$-domain of $\eta$ will be simply 
\[
U=(b-a)\widetilde{U}+a.
\]

\section{Comments}
In this work, we proposed a numerical framework to approximate the
$\mu$-domain associated with a given distribution. The assumption on the sequence $\mu_n$ is the weakest ever. The theoretical
results were complemented with explicit constructions and convergence
rates, as well as practical implementation strategies involving the
Hilbert transform. Numerical simulations demonstrated the robustness
of the method across various distributions, validating the effectiveness
and generality of our approach. Note that our approach works perfectly
also with the construction given by Boudabra and Markowsky in \cite{Boudabra2020}.
We think that the following two questions are worth investigating
: 
\begin{itemize}
\item Is the rate of $\frac{1}{n}$ the optimal one for any possible approximation technique? Note that in our method, it can be checked that $\frac{1}{n}$ is optimal (uniform distribution for example)
\item What is the variance of $q-q_{n}$ and so the MSE? What scheme would
give the least variance? 
\end{itemize}
\bibliographystyle{plain}
\bibliography{NumericalApproach}

\end{document}